\documentclass[a4paper,12pt]{amsart}
\usepackage{amsfonts,amsmath,amssymb,amsthm,amsaddr}
\usepackage{latexsym}
\usepackage{eucal}
\usepackage[dvips]{graphics}
\usepackage{mathabx}
\usepackage[english]{babel}
\usepackage{epsfig}
\usepackage[usenames]{color}
\usepackage[colorlinks=true, linkcolor=red, citecolor=blue]{hyperref}
\newcommand{\nwc}{\newcommand}
\nwc\eps{\varepsilon}

\newtheorem{theorem}{Theorem}[section]
\newtheorem{proposition}[theorem]{Proposition}
\newtheorem{lemma}[theorem]{Lemma}
\newtheorem{corollary}[theorem]{Corollary}

\theoremstyle{remark}
\newtheorem*{remarks}{Remarks}

\usepackage{mathrsfs}
\def\supp{\operatorname{{supp}}}

\providecommand{\abs}[1]{\left\lvert#1\right\rvert}

\numberwithin{equation}{section}

\makeatletter
\@namedef{subjclassname@2010}{\textup{2020} Mathematics Subject Classification}
\makeatother

\begin{document}

\title[Controllability of Kawahara equation]{Approximation theorem for the Kawahara operator and its application in control theory}

\author[Capistrano-Filho]{Roberto de A.  Capistrano--Filho*}
\thanks{*Corresponding author.}
\address{
Departamento de Matem\'atica, Universidade Federal de Pernambuco\\
Cidade Universit\'aria, 50740-545, Recife (PE), Brazil\\
Email address: \normalfont\texttt{roberto.capistranofilho@ufpe.br}}



\author[de Sousa]{Luan S. de Sousa}
\address{
Departamento de Matem\'atica, Universidade Federal de Pernambuco\\
Cidade Universit\'aria, 50740-545, Recife (PE), Brazil\\
Email address: \normalfont\texttt{luan.soares@ufpe.br}}


\author[Gallego]{Fernando A. Gallego}
\address{
Departamento de Matemática, \\ Universidad Nacional de Colombia (UNAL), \\ Cra 27 No. 64-60, 170003, Manizales, Colombia\\
Email address: \normalfont\texttt{fagallegor@unal.edu.co}}

\thanks{Capistrano--Filho was supported by CNPq grants numbers  307808/2021-1, 401003/2022-1 and 200386/2022-0, CAPES  grants numbers 88881.311964/2018-01 and 88881.520205/2020-01,  and MATHAMSUD 21-MATH-03. Gallego was partially supported by MATHAMSUD 21-MATH-03, Hermes Unal project nro 57774, and Hermes Unal 55686.}


\begin{abstract}
Control properties of the Kawahara equation are considered when the equation is posed on an unbounded domain.  Precisely,  the paper's main results are related to an approximation theorem that ensures the exact (internal) controllability in $(0,+\infty)$. Following \cite{Rosier}, the problem is reduced to prove an approximate theorem which is achieved thanks to a global Carleman estimate for the Kawahara operator.
\end{abstract}
\maketitle

\section{Introduction}
\subsection{Problem set}
Our main focus in this work is to investigate the control property for the Kawahara equation \cite{Hasimoto1970,Kawahara}
\begin{equation}\label{fda1}
	u_{t}+u_{x}+u_{xxx}-u_{xxxxx}+uu_{x}=0
\end{equation}
which is a dispersive PDE describing numerous wave phenomena such as magneto-acoustic waves in a cold plasma \cite{Kakutani}, the propagation of long waves in a shallow liquid beneath an ice sheet \cite{Iguchi}, gravity waves on the surface of a heavy liquid \cite{Cui}, etc. In the literature, this equation is also referred to as the fifth-order KdV equation \cite{Boyd}, or singularly perturbed KdV equation \cite{Pomeau}.

Some valuable efforts in the last years focus on the analytical and numerical methods for solving \eqref{fda1}. These methods include the tanh-function method \cite{Berloff}, extended tanh-function method
\cite{Biswas}, sine-cosine method \cite{Yusufoglu}, Jacobi elliptic functions method \cite{Hunter}, direct algebraic method \cite{Polat}, decompositions methods \cite{Kaya}, as well as the variational iterations and homotopy
perturbations methods \cite{Jin}. 

Due to this recent advance, previously mentioned, other issues for the study of the Kawahara equation appear. For example, we can cite the control problems, which is our motivation. Precisely, we are interested in proving control results for the Kawahara operator in an unbounded domain. It is well known that the first result with a ``kind" of controllability for the Kawahara equation
\begin{equation}\label{kaw} 
u_t+u_x+u_{xxx}-u_{xxxxx}=f(t,x), \quad (t,x)\in\mathbb{R}^{+}\times(0,\infty),
\end{equation}  
was proposed recently by the authors in \cite{CaSoGa}.  It is important to point out that in \cite{CaSoGa}, the authors are not able to prove that solutions of \eqref{kaw} satisfy the exact controllability property 
\begin{equation}\label{ex}
u(T,x)=u_{T} \quad x\in(0,\infty).
\end{equation}
Instead of this, they showed that solutions of the Kawahara equations satisfy an integral condition. 

To fill this gap in providing a study of the exact boundary controllability of \eqref{kaw} in an unbounded domain, this paper aims to present a way that may be seen as a first step in the knowledge of control theory for the system \eqref{kaw} on unbounded domains since the results proved in \cite{CaSoGa}, can not recover \eqref{ex}.  So, our aim in this manuscript is to present an answer to the following question:

\vspace{0.2cm}
\noindent \textbf{Problem $\mathcal{A}:$}\textit{Is there a solution to the system \eqref{kaw} satisfying \eqref{ex}? Or, equivalently, Is the solution of the system \eqref{kaw} exact controllable in the unbounded domain $(0,+\infty)$?}

\subsection{Historical background}
Stabilization and control problems on the bounded domain have been studied in recent years for the Kawahara equation. The first work concerning the stabilization property for the Kawahara equation in a bounded domain $(0, T) \times (0,L)$, is due to Capistrano--Filho \textit{et al.} in \cite{CaKawahara}.  In this article, the authors were able to introduce an internal feedback law and, considering general nonlinearity $u^pu_x$, $p\in [1,4)$, instead of $uu_x$, to show that under the effect of the damping mechanism the energy associated with the solutions of the system decays exponentially. 

Concerning the internal control problems we can cite as pioneer works the Zhang and Zhao articles \cite{zhang, zhang1}.  In both works the authors considered the Kawahara equation in a periodic domain $\mathbb{T}$ with a distributed control of the form \[ f(t,x)=(Gh)(t,x):= g(x)(  h(t,x)-\int_{\mathbb{T}}g(y) h(t,y) dy), \] where $g\in C^\infty (\mathbb T)$ supported in $\omega\subset\mathbb{T}$ and $h$ is a control input.  Still related to internal control issues, Chen \cite{MoChen} presented results considering the Kawahara equation posed on a bounded interval with a distributed control $f(t,x)$ and homogeneous boundary conditions. She showed the result by taking advantage of a Carleman estimate associated with the linear operator of the Kawahara equation with an internal observation. With this in hand, she was able to get a null controllable result when $f$ is effective in a $\omega\subset(0,L)$.  

As the results obtained by Chen in \cite{MoChen} do not answer all the issues of internal controllability, in a recent article \cite{CaGo} the authors closed some gaps left in \cite{MoChen}. Precisely, considering the Kawahara model with an internal control $f(t,x)$ and homogeneous boundary conditions, the authors can show that the equation in consideration is exactly controllable in $L^2$-weighted Sobolev spaces and, additionally, the Kawahara equation is controllable by regions on $L^2$-Sobolev space, for details see \cite{CaGo}.

Recently,  a new tool to find control properties for the Kawahara operator was proposed in \cite{CaSo, CaSoGa}.  First, in \cite{CaSo}, the authors showed a new type of controllability for a Kawahara equation, what they called the overdetermination control problem. Precisely, they can find a control acting at the boundary that guarantees that the solution of the problem under consideration satisfies an integral condition. In addition, when the control acts internally in the system, instead of the boundary, the authors proved that this condition is also satisfied. These problems give answers that were left open in \cite{CaGo} and present a new way to prove boundary and internal controllability results for the Kawahara operator.  After that, in \cite{CaSoGa}, the authors extend this idea for the internal control problem for the Kawahara equation on unbounded domains.  Precisely, under certain hypotheses over the initial and boundary data,  they can prove that an internal control input exists such that solutions of the Kawahara equation satisfy an integral overdetermination condition considering the Kawahara equation posed in the real line, left half-line, and right half-line.  

\subsection{Main results} With this background in hand, as mentioned before,  our main goal is to answer the Problem $\mathcal{A}$. To do that, we first prove two main results which are the key to giving some position of the controllability properties for the Kawahara operator on an unbounded domain.

Let us introduce some notations. For $L>0$ and $T>0$ let $Q_{T}=\{(x,t)\in
(-L,L)\times(0,T)\subset\mathbb{R}^{2}\},$ be a bounded rectangle.  From now on, for the sake of brevity, we shall write $P$ for the operator 
\begin{equation} \label{P}
P=\partial_t+\partial_x+\partial_x^3-\partial^5_x
\end{equation} 
with domain 
\begin{equation}\label{D_P}
\mathcal{D}(P)=L^2(0,T;H^5(-L,L)\cap H_0^2(-L,L))\cap H^1(0,T;L^2(-L,L)).
\end{equation} 

Our first result is related to a Carleman estimate for the Kawahara operator being precise, for $f\in L^2(0,T;L^2(-L,L))$ and $q_0\in L^2(-L,L),$ the operator $P q=f$, where $P$ is defined by \eqref{P} with domain \eqref{D_P}.  So, the first result is devoted to proving a global Carleman estimate.
\begin{theorem}\label{main1} There exist constants $s_0=s_{0}(L,T)>0$ and $\tilde{C}= \tilde{C}(L,T)>$ such that for any $q\in\mathcal{D}(P)$ and all $s \geq s_{0}$, one has
\begin{equation}\label{DGC}
\begin{split}
\int_{0}^{T}\int_{-L}^{L}\left\{(s\varphi)^9|q|^2+(s\varphi)^{7}|q_x|^2 + (s\varphi)^{5}|q_{xx}|^{2}+(s\varphi)^{3}|q_{xxx}|^{2}+ s\varphi|q_{xxxx}|^{2}\right\}e^{-2s\varphi}dxdt \\ \leq C\int_{0}^{T}\int_{-L}^{L}|f|^{2}e^{-2s\varphi}dxdt.
\end{split}
\end{equation} 
\end{theorem}

As a consequence of the previous Carleman estimate,  the second main result of the manuscript gives us an approximation theorem, which is the key point to prove the exact controllability for the operator $P$ posed on unbounded domain and, in this case,  to answer the Problem $\mathcal{A}$.

\begin{theorem}\label{main2} Let  $n \in \mathbb{N}\backslash \{0,1\},$ and $t_{1}, t_{2}$ and $T$ real number such that $0 < t_{1} < t_{2} < T.$ Let us consider $ u \in L^{2}((0,T)\times(-n,n))$ such that $$Pu=0\quad \text{in}\quad (0,T)\times(-n,n),$$ with $\supp \ u \subset [t_{1}, t_{2}]\times(-n,n)$. Let $0 < \epsilon < min(t_{1}, T - t_{2}),$ then there exists $v \in L^{2}((0,T)\times(-n-1,n+1))$ satisfying
\begin{equation}\label{4.17}
Pv= 0 \ \mbox{in} \ (0,T)\times(-n-1, n+1),
\end{equation} 
\begin{equation}\label{4.18}
\supp \ v \subset [t_{1} - \epsilon, t_{2} + \epsilon] \times (-n-1,n+1),
\end{equation} 
 and
\begin{equation}\label{4.19}
\|v - u\|_{L^{2}((0,T)\times(-n+1,n-1))} < \epsilon.
\end{equation}
\end{theorem}

Finally, the previous result helps to show the third main result of the manuscript, giving a positive answer for the exact controllability problem.

\begin{theorem}\label{main3}
Given $T, \epsilon$ an $s$ real numbers with $0 < \epsilon < \frac{T}{2}$ and  $s \in\left(-\frac{7}{4}, \frac{5}{2}\right) \backslash\left\{\frac{1}{2}, \frac{3}{2}\right\}$.  Let $u_{0}, u_{T} \in H^{s}(0,+\infty)$, thus, there exists a function 
\begin{equation*}
u \in L^{2}_{\mbox{loc}}([0,T]\times(0,+\infty)) \cap C([0,\epsilon];H^{s}(0,+\infty)) \cap C([T-\epsilon,T];H^{s}(0,+\infty)
\end{equation*}
solution of 
\begin{equation}\label{PP}
\begin{cases}
u_{t} +  u_{x} +  u_{xxx} - u_{xxxxx}= 0 & \mbox{in} \ \mathcal{D}'((0,T)\times(0,+\infty)), \\ 
u(0,x)= u_{0} & \mbox{in} \ (0,+\infty), 
\end{cases}
\end{equation} 
satisfying  $u(T,x) = u_{T}$  in $(0,+\infty).$
\end{theorem}

%

\subsection{Final comments and paper's outline}
The results in this manuscript gave a necessary first step to the improvement of the control properties for the Kawahara operator.  Let us comment on this in the following remark.
\begin{remarks} The following remarks are worth mentioning:
\begin{itemize}
\item[i.] From our knowledge, our results are the first ones for the Kawahara operator posed on an unbounded domain. 
\item[ii.]Note that the Carleman estimate proved in \cite{MoChen} is local which differs from the Carleman estimates shown in Theorem \ref{main1}.    
\item[iii.] This work is the first one to prove an approximation theorem, that is, Theorem \ref{main2}, for the Kawahara operator \eqref{P}.
\item[iv.] In the context of the Kawahara operator, there is one work \cite{CaSoGa} which is limited from a control point of view since the solutions satisfy an integral condition instead of  \eqref{ex}. Thus, Theorem \ref{main3} provides progress in the control theory for this operator in an unbounded domain thanks to the fact that solutions of \eqref{PP} satisfy the exact controllability condition  \eqref{ex}.
\item[v.] It is important to point out that the strategy applied in our work was already applied for the Korteweg--de Vries (KdV) equation \cite{Rosier} and the KdV-Burgers equation \cite{Gallego}. In both cases, a Carleman estimate is derived following Fursikov–Imanuvilov’s approach \cite{FI}.
\item[vi.] The Kawahara equation \eqref{PP} is a higher-order KdV equation, here called the Kawahara equation or fifth-order KdV equation. So, for this operator, some extra difficulties appear. The first main difficulty is to prove a Carleman estimate. Note that we can not directly apply the estimates proposed in \cite[Proposition 3.1]{Rosier} or \cite[Lemma 2.4]{Gallego}, since we have a fifth-order equation and more terms (included traces) need to be controlled (see Section \ref{Sec2}).
\item[vii.] Concerning the exact controllability result, Theorem \ref{main3}, note that the restriction in $s$ for the space $H^s$ is required, this is because the well-posedness on an unbounded domain for the system \eqref{PP} follows if $s \in\left(-\frac{7}{4}, \frac{5}{2}\right) \backslash\left\{\frac{1}{2}, \frac{3}{2}\right\}$, which not happens in \cite{Rosier,Gallego}. On the other hand, since we have a more strong well-posedness solution borrowed from \cite{MC}, we do not need the $L^2$ space with weight as in \cite[Theorem 1.3]{Rosier} and \cite[Theorem 1.2]{Gallego}, for example.
\item[viii.] Summarizing, our result gives new results for the Kawahara operator (higher-order KdV equation) in the following sense:
\begin{enumerate}
\item  New global Carleman estimate;
\item Approximation theorem;
\item Exact controllability in $H^s$, when  $s \in\left(-\frac{7}{4}, \frac{5}{2}\right) \backslash\left\{\frac{1}{2}, \frac{3}{2}\right\}$.
\end{enumerate}
\end{itemize}
\end{remarks}

The remainder of the paper is organized as follows. In Section \ref{Sec1}, we present auxiliary results which are paramount to show the main results of the article.  In Section \ref{Sec2}, we present the global Carleman estimate, that is, we will show Theorem \ref{main1}. Section \ref{Sec3} is devoted to giving applications of the Carleman estimate, precisely,  we will provide an approximation Theorem \ref{main2}. Finally,  in Section \ref{Sec4},  we will answer the Problem $\mathcal{A}$ using the approximation theorem, i.e.,  we present the proof of Theorem \ref{main3}. 

\section{Preliminaries} \label{Sec1}
\subsection{Auxiliary lemma}
In this subsection, we will prove an auxiliary result that will put us in a position to apply it to prove the main results of the article. For this propose, observe that the operator $P$ generates a $C_0$--semigroup of contractions ${S_{L}(t)}_{t\geq0}$ on $L^2(-L, L)$ (see for instance \cite{CaKawahara}) which be denoted now on by $S_{L}(\cdot)$. With this in hand, the next lemma holds.
\begin{lemma}\label{4.2} Consider $l_{1}, l_{2}, L, t_{1}, t_{2}$ and $T$ be number such that $0 < l_{1} < l_{2} < L $ and $0 < t_{1} < t_{2} < T$.  Let $ u \in L^{2}((0,T)\times(-l_{2}, l_{2}))$ be such that 
\begin{equation*}
Pu= 0 \ \mbox{in} \ (0,T)\times(-l_{2}, l_{2}) \quad \mbox{and} \quad \supp \ u \subset [t_{1}, t_{2}]\times(-l_{2}, l_{2}).
\end{equation*}
Let $\eta > 0$ and $\delta > 0$, with $2\delta < \min(t_{1}, T - t_{2})$ be given. Then there exist $v_{1}, v_{2} \in L^{2}(-L,L)$ and $v \in L^{2}((0,T)\times(-L, L))$ such that
\begin{equation}\label{4.6}
Pv= 0 \ \mbox{in} \ (0,T)\times(-L,L),
\end{equation}
\begin{equation}\label{4.7}
v(t,\cdot)= S_{L}(t - t_{1} + 2\delta)v_{1}, \ \mbox{for} \ t_{1} - 2\delta < t < t_{1} - \delta,
\end{equation}
\begin{equation}\label{4.8}
v(t,\cdot)= S_{L}(t - t_{2} + \delta)v_{2}, \ \mbox{for} \ t_{2} + \delta < t < t_{2} + 2\delta
\end{equation}
and
\begin{equation*}
\|v - u\|_{L^{2}((t_{1} - 2\delta, t_{2} + 2\delta)\times(-l_{1}, l_{1}))} < \eta.
\end{equation*}
\end{lemma}
\begin{proof}
Remember that $Q_T=(0, T)\times(-L,L)$, $P$ is defined by \eqref{P}-\eqref{D_P} and pick $Q_{\delta}= (t_{1} - 2\delta, t_{2} + 2\delta)\times(-l_{1}, l_{1}).$ By a smoothing process via convolution and multiplying the regularized function by a cut-off function of $x$, we have a function $u' \in \mathcal{D}(\mathbb{R}^{2}),$ such that
\begin{equation}\label{4.10}
\begin{cases}
\supp \ u' \subset [t_{1} - \delta, t_{2} - \delta] \times [-l_{2}, l_{2}], \\ 
Pu'= 0 \ \mbox{in} \ (0,T)\times(-l_{1}, l_{1}),  \quad \text{and}\\
\|u' - u\|_{L^{2}((0,T)\times(-l_{1}, l_{1}))} < \frac{\eta}{2}.
\end{cases}
\end{equation} 
Consider the following set 
$$\mathcal{E}= \{v \in L^{2}(Q_T); \exists \ v_{1}, v_{2} \in L^{2}(-L,L) \ \mbox{such that} \ \eqref{4.6}, \eqref{4.7}\ \mbox{and} \ \eqref{4.8} \ \mbox{hold true} \}.$$
Note that this lemma is proved if we may find $v\in\mathcal{E}$ such that 
$$\|v - u'\|_{L^{2}(Q_{\delta})} < \frac{\eta}{2}.$$
It follows by the following trivial inequality 
\begin{equation*}
\begin{split}
\|v - u\|_{L^{2}(Q_{\delta})} \leq& \|v - u'\|_{L^{2}(Q_{\delta})} + \|u' - u\|_{L^{2}(Q_{\delta})}\\
<&\|v - u'\|_{L^{2}(Q_{\delta})} + \frac{\eta}{2}.
\end{split}
\end{equation*}
So we achieve the proof if we prove that $u' \in \overline{\mathcal{E}}= (\mathcal{E}^{\perp})^{\perp},$ where the closure and the orthogonal complement
are taken in the space $L^{2}(Q_{\delta}).$ For a fix function $g \in \mathcal{E}^{\perp} \subset L^{2}(Q_{\delta})$ we should prove that the following holds
\begin{equation}\label{main}
(u',g)_{L^{2}(Q_{\delta})}= 0.
\end{equation}
Before presenting the proof of \eqref{main}, we claim the following. 

\vspace{0.2cm}
\noindent\textbf{Claim 1.} Let $\mathcal{T}= \{ \varphi \in C^{\infty}(\mathbb{R}^{2}); \supp \ \varphi \subset [t_{1} - \delta, t_{2} + \delta]\times{\mathbb{R}}\}.$ So, there exists $C > 0$ such that
\begin{equation}\label{4.11}
|(\varphi, g)_{L^{2}(Q_{\delta})}| \leq C\|P\varphi\|_{L^{2}(Q_T)},
\end{equation}
for all $ \varphi \in \mathcal{T}.$
\vspace{0.2cm}

In fact,  pick $\varphi \in \mathcal{T}$ and define 
$$\psi(t)= \int_{0}^{t}S_{L}(t - s)P \varphi(s)ds, $$
for $0\leq t \leq T$,  that is, $\psi$ is strong solution of the boundary initial-value problem 
\begin{equation*}
\begin{cases}
P\psi=0,&\text{in }Q_{T}\text{,}\\
\psi(t,-L)=\psi(t,L), \quad \psi_{x}(t,-L)= \psi_{x}(t,L), \quad \psi_{xx}(t,-L)= \psi_{xx}(t,L), & t\in[0,T], \\\psi_{xxx}(t,-L)= \psi_{xxx}(t,L), \quad \psi_{xxxx}(t,-L)= \psi_{xxxx}(t,L), & t\in[0,T], \\
\psi(0,\cdot)= 0,&\text{in }[-L,L].
\end{cases}
\end{equation*}
Thanks to this fact,  $v= \psi - \varphi\in\mathcal{E}$, observe that \eqref{4.7} and \eqref{4.8} is verified with $v_1=0$ and $v_2=\psi(t_2+\delta)$, hence 
$$(v, g)_{L^{2}(Q_{\delta})}= (\psi -\varphi, g)_{L^{2}(Q_{\delta})}= 0.$$
On the other hand,  we have 
\begin{equation*}
\|\psi(t)\|_{L^{2}(-L,L)} \leq \|P\varphi\|_{L^{1}(0,t; L^{2}(-L,L))} \leq \sqrt{T}\|P\varphi\|_{L^{2}(Q_T))},
\end{equation*}
for all $t\in[0,T]$, and therefore 
\begin{equation*}
|(\varphi,g)_{L^{2}(Q_{\delta})}|= |(\psi,g)_{L^{2}(Q_{\delta})}| \leq T\|g\|_{L^{2}(Q_{\delta})}\|P\varphi\|_{L^{2}(Q_T)},
\end{equation*}
showing Claim 1.  We also need the following claim.

\vspace{0.2cm}
\noindent\textbf{Claim 2.}  There exists a function $\omega \in L^{2}(Q_T)$ such that 
\begin{equation}\label{4.12}
(\varphi, g)_{L^{2}(Q_{\delta})}= (P\varphi, \omega)_{L^{2}(Q_T)}, 
\end{equation}
for all $ \varphi \in \mathcal{T}.$
\vspace{0.2cm}

Indeed,  let $\mathcal{Z}= \{(P\varphi)\bigr|_{Q}; \varphi \in \mathcal{T}\}$ and define the map $ \Lambda:\mathcal{Z} \longrightarrow \mathbb{R}$ by 
$$ \Lambda(\zeta)= (\zeta, g)_{L^{2}(Q_{\delta})}.$$
First, note that for any  $\zeta \in \mathcal{Z},$ if $\zeta= (P\varphi_{1})\bigr|_{Q_T}= (P\varphi_{2})\bigr|_{Q_T}$,  for two functions $\varphi_{1}, \varphi_{2} \in \mathcal{T},$ we have using claim 1 that $\varphi_{1}-\varphi_{2} \in \mathcal{E},$ hence $(\varphi_{1}-\varphi_{2}, g)_{L^{2}(Q_{\delta})}= 0.$ Thus,  $\Lambda$ is well defined.   Consider $H$ the closure of $\mathcal{Z}$ in $L^{2}(Q).$ Due to \eqref{4.11},  using the Hahn-Banach theorem, we may extend $\Lambda$ to $H$ in such way that $\Lambda$ is a continuous linear form on $H$. Thus, it follows from Riesz representation theorem that  there exists $\omega \in H$ such that 
$$\Lambda(\zeta)= (\zeta, \omega)_{L^{2}(Q_T)}, \ \forall \zeta \in H,$$
and so \eqref{4.12} follows,  and the proof of Claim 2 is finished. 

Finally, let us prove \eqref{main}. To do it, consider the extensions of  $g$ and $\omega$ in $\mathbb{R}^{2}$ given by
$$\tilde{g}(t,x)=0, \ \mbox{for} \ (t,x) \in \mathbb{R}^{2}\backslash Q_{\delta}$$
and
$$\tilde{\omega}(t,x)=0, \ \mbox{for} \ (t,x) \in \mathbb{R}^{2}\backslash Q_T,$$
respectively.  Taking $\Omega = (t_{1} - \delta, t_{2} - \delta)\times\mathbb{R}$,  let  $\varphi \in \mathcal{D}(\Omega) \subset \mathcal{T}.$ So, we have that 
$$(\varphi, g)_{L^{2}(Q_{\delta})}= (\varphi,\tilde{g})_{L^{2}(\Omega)} \quad \mbox{and} \quad (P\varphi, \omega)_{L^{2}(Q_T)}= (P\varphi, \tilde{\omega})_{L^{2}(\Omega)},$$
therefore,  using \eqref{4.12},  we get
\begin{equation*}
\langle P^{*}(\tilde{\omega}), \varphi \rangle_{\mathcal{D}'(\Omega), \mathcal{D}(\Omega)}= \langle\tilde{g}, \varphi\rangle_{\mathcal{D}'(\Omega), \mathcal{D}(\Omega)},
\end{equation*}
so  $P^{*}(\tilde{\omega})= \tilde{g}$ in $\mathcal{D}'(\Omega)$ and
$$P^{*}(\tilde{\omega})= 0, \ \mbox{for} \ t_{1} - \delta < t < t_{2} + \delta \ \mbox{and} \ |x| > l_{1}.$$
Since 
$$\tilde{\omega}(t,x)=0, \ \mbox{for} \ t_{1}- \delta < t < t_{2} - \delta \ \mbox{and} \ |x| > L,$$
Holmgren's uniqueness theorem (see e. g. \cite[Theorem 8.6.8]{Horm}) ensures that 
$$\tilde{\omega}(t,x)= 0, \ \mbox{for} \ t_{1} - \delta < t < t_{2} + \delta \ \mbox{and} \ |x| > l_{1}.$$
Lastly, due to \eqref{4.12} and \eqref{4.10},  we conclude that
$$(u',g)_{L^{2}(Q_{\delta})}= (Pu', \omega)_{L^{2}(Q)}=(Pu', \omega)_{L^{2}((t_{1} - \delta, t_{2} + \delta)\times(-l_{1}, l_{1}))}=0,$$
finishing the proof.
\end{proof}

\subsection{Observability inequality \textit{via} Ingham inequality} Given a family $\Omega=(\omega_k)_{k\in K}:=\{\omega_k : k\in K\}$ of real numbers, we consider functions of the form $
\sum_{k\in K}c_ke^{i\omega_kt}$ with square summable complex coefficients $(c_k)_{k\in K}:=\{c_k : k\in K\}$, and we investigate the relationship between the quantities
\begin{equation*}
\int_I\abs{\sum_{k\in K}c_ke^{i\omega_kt}}^2\ dt
\quad \text{and} \quad
\sum_{k\in K}\abs{c_k}^2,
\end{equation*}
where $I$ is some given bounded interval. In this work, the following version of the Ingham-type theorem will be used.

\begin{theorem}\label{tgeneralizedingham}
 Let $\{\lambda_{k}\}$ be a family of real numbers,
satisfying the uniform gap condition
\begin{equation*}
\gamma = \inf_{k \neq n} | \lambda_{k} - \lambda_{n}| > 0
\end{equation*}
and set
\begin{equation*}
\gamma' = \sup_{A \subset K} \inf_{k, n \in K \setminus A} |\lambda_{k} -\lambda_{n}| > 0
\end{equation*}
where $A$ rums over the finite subsets of $K$. If I is a bounded interval of length $|I| \geq \frac{2 \pi}{\gamma'}$, then there exist positive constants $A$ and $B$ such that
\begin{equation*}
A\sum_{k \in K} |c_{k}|^{2} \leq \int_{I} |f(t)|^{2} dt \leq B \sum_{k \in K} |c_{k}|^{2}
\end{equation*}
for all functions given by the sum $f(t) = \sum_{ k \in K} c_{k}e^{i\lambda_{k}t}$ with square-summable complex coefficients $c_{k}$.
\end{theorem}
\begin{proof}
See \cite[Theorem 4.6]{KomLor2005}.
\end{proof}


Now on, consider the operator $A: D(A) \subset L^{2}(-L,L) \longrightarrow L^{2}(-L,L),$ defined by $A(u)= - u_{x} -  u_{xxx} + u_{xxxxx}$, with  $$D(A)=\{v \in H^{5}(-L,L); v(-L)= v(L), v_{x}(-L)= v_{x}(L),..., v_{xxxx}(-L)= v_{xxxxx}(L) \}.$$ In what follows $S_{L}$ will denote the unitary group in $L^{2}(-L,L)$ generated by the operator $A$, using Stone theorem.   With this in hand, pick $e_{n}= \frac{1}{\sqrt{2L}}e^{in\frac{\pi}{L}x}$ for $ n \in \mathbb{Z}$. So, $e_n$ is an eigenvector for $A$ associated with the eigenvalue $\omega_{n} = i\lambda_{n}$, with 
\begin{equation*}
\lambda_{n}= \biggl(\frac{n \pi}{L}\biggr)^{5} + \biggl(\frac{n \pi}{L}\biggr)^{3} - \frac{n \pi}{L}.
\end{equation*}
If  $u_{0} \in L^{2}(-L,L)$ is any complex function, we decomposed as $u_{0}= \sum_{n \in \mathbb{Z}}^{}c_{n}e_{n},$ so we have for every $t\in\mathbb{R}$
$$
S_{L}(t)u_{0}= \sum_{n \in \mathbb{Z}}^{}e^{i\lambda_{n}t}c_{n}e_{n}.
$$
We are now in a position to prove an observability result. 
\begin{proposition}\label{OI_I} Let $l, L,$ and $T$ be positive number such that $l < L.$ Then there exists a constant positive $C$ such that for every $u_{0} \in L^{2}(-L,L)$,  denoting $u= S_{L}(.)u_{0},$ we get
\begin{equation}\label{4.14}
\|u_{0}\|_{L^{2}(-L,L)} \leq C\|u\|_{L^{2}((0,T)\times(-l,l))}.
\end{equation}
Therefore, 
\begin{equation}\label{4.15}
\|u\|_{L^{2}((0,T)\times(-L,L))} \leq \sqrt{T} C\|u\|_{L^{2}((0,T)\times(-l,l))}.
\end{equation}
\end{proposition}
\begin{proof}
Pick $T' \in (0, \frac{T}{2})$ and $\gamma > \frac{\pi}{T'}$.  Let $N\in\mathbb{N}$ such that
$$\lambda_{N}-\lambda_{-N}=2 \lambda_{N} \geq \gamma \text { and }(n \in \mathbb{Z},|n| \geq N) \Rightarrow \lambda_{n+1}-\lambda_{n} \geq \gamma.$$
By Ingham's inequality, see Theorem \ref{tgeneralizedingham},  there exists a constant $C_{T'} > 0$ such that for every sequence $(a_{n})_{|n|>N}$ of complex numbers, with  $a_{n}=0$, for all  $n \in \mathbb{Z}; |n|< N$,  the following inequality is verified
\begin{equation}\label{4.16}
\sum_{|n|\geq N}^{}|a_{n}|^{2} \leq C_{T'}\int_{0}^{2T'}\biggl|\sum_{|n|\geq N}^{}a_{n}e^{i\lambda_{n}t}\biggr|^{2}dt
\end{equation}
Let $\mathcal{Z}_{n} = Span(e_{n})$ for $n \in \mathbb{Z}$ and $\mathcal{Z}= \oplus_{n \in \mathbb{Z}} \mathcal{Z}_{n} \subset L^{2}(-L,L).$ Let us now define the following seminorm $p$ in $\mathbb{Z}$ by 
$$p(u)= \biggl(\int_{-l}^{l}|u(x)|^{2}dx\biggr)^{\frac{1}{2}}dt, \ \forall u \in \mathcal{Z}.$$
In this case,  $p$ is a norm in each $\mathcal{Z}_{n}.$ By other hand, if  $u_{0} \in \mathcal{Z} \cap (\oplus_{|n|<N}\mathcal{Z})^{\perp}$,  we can rewrite  $u_{0}$ in the following way $$u_{0}= \sum_{|n|> N}^{}c_{n}e_{n},$$ with $c_{n}=0$ for $|n|$ large enough.  Thus, applying \eqref{4.16} with $a_{n}= \frac{c_{n}}{\sqrt{2L}}e^{i(\lambda_{n}T'+n\frac{\pi}{L}x)}$  and integrating in $(-l,l)$ we get
\begin{equation*}
2l\sum_{|n|\geq N}^{}\frac{|c_{n}|^{2}}{2L} \leq C_{T'}\int_{-l}^{l}\int_{0}^{2T'}\biggl|\sum_{|n|\geq N}^{}e^{i\lambda_{n}t}c_{n}e_{n}(x)\biggr|^{2}dtdx.
\end{equation*}
Therefore,  Fubini's theorem ensures that 
\begin{equation*}
\|u_{0}\|_{L^{2}(-L,L)} \leq \frac{L}{l}C_{T'}\int_{0}^{2T'}p(S_{L}(t)u_{0})^{2}dt.
\end{equation*}
Finally, for $u_{0} \in L^{2}(-L,L)$, we have
\begin{equation*}
\int_{0}^{2T'}p(S_{L}(t)u_{0})^{2}dt \leq \|S_{L}(.)u_{0}\|^{2}_{L^{2}((0,2T')\times(-L,L))}= 2T'\|u_{0}\|_{L^{2}(-L,L)}.
\end{equation*}
Thanks to the fact that $2T' < T$,  follows from \cite[Theorem 5.2]{Vilmos} that there exists a positive constant, still denoted by $C$, such that \eqref{4.14} is verified for all $z_{0} \in \mathcal{Z}$ and the general case, that is, for all $u_{0} \in L^{2}(-L,L)$, follows by a density argument, showing the result.
\end{proof}


\section{Global Carleman estimate}\label{Sec2} 

Consider $T$ and $ L > 0$ to be positive numbers.  Pick any function $\psi\in C^8{[-L,L]}$ with
\begin{equation}\label{ca2}
\begin{split}
\psi>0 \text { in }[-L, L] ; \quad \psi^{\prime}(-L)>0 ; \quad \psi^{\prime}(L)>0,
\psi'' < 0 \quad \text{and}\quad |\psi_{x}| > 0 \text{ in } [-L,L].
\end{split}
\end{equation} 
Let $u = e^{-s\varphi}q$, $\omega = e^{-s\varphi}P(e^{s\varphi}u)$ and $\varphi(t,x)=\frac{\psi(x)}{t(T-t)}$. Straightforward computations show that
\begin{equation}\label{omega ppp}
\omega= L_1(u)+L_2(u),
\end{equation}
with 
\begin{equation*}
\begin{split}
L_{1}(u)&= Au + C_{1}u_{xx} + Eu_{4x} , \\ 
L_{2}(u)&= Bu_{x} + C_{2}u_{xx} + Du_{xxx} + u_{t} - u_{5x}.
\end{split}
\end{equation*}
Here
\begin{equation*}
\begin{split}
A=& s(\varphi_{t} + \varphi_{x} + \varphi_{xxx} - \varphi_{5x}) - s^{2}(10\varphi_{xx}\varphi_{xxx} - 3\varphi_{x}\varphi_{xx} + 5\varphi_{x}\varphi_{4x})
\\ &- s^{3}(15\varphi_{x}\varphi^{2}_{xx} + 10\varphi^{2}_{x}\varphi_{xxx} -\varphi^{3}_{x}) - s^{4}10\varphi^{3}_{x}\varphi_{xx} - s^{5}\varphi^{5}_{x},\\
B =&  + s(3\varphi_{xx} - 5\varphi_{4x}) - s^{2}(15\varphi^{2}_{xx} + 20\varphi_{x}\varphi_{xxx} - 3\varphi^{2}_{x}) - s^{3}30\varphi^{2}_{x}\varphi_{xx} - s^{4}5\varphi^{4}_{x},\\
C_1=& s(3\varphi_{x} - 10\varphi_{xxx}) - s^{3}10\varphi^{3}_{x}\\
C_2=&C_{2}= - s^{2}30\varphi_{x}\varphi_{xx}\\
D= & -s10\varphi_{xx} - s^{2}10\varphi^{2}_{x},\\
E=& -s5\varphi_{x}.
\end{split}
\end{equation*}
On the other hand
$
\|\omega\|^{2}=\left\|L_{1}( u)\right\|^{2}+\left\|L_{2}(u)\right\|^{2}+2\left(L_{1}(u), L_{2} (u)\right)
$
where $$(u, v)=\int_{0}^{T} \int_{-L}^{L} u v \mathrm{~d} x \mathrm{~d} t$$ and $\|\omega\|^{2}=(\omega, \omega)$.   With this in hand,  we can prove a global Carleman estimate for the Kawahara equation 
\begin{equation*}
\begin{cases}
u_{t}+ u_{x}+ u_{xxx}-u_{xxxxx}=0&(x,t)\in Q_{T}\text{,}\\
u\left(  -L,t\right)  =u\left(  L,t\right)  =u_{x}\left(  -L,t\right)
=u_{x}\left(  L,t\right)  =u_{xx}\left(  L,t\right)  =0&
t\in\left(  0,T\right)  \text{,}\\
u\left(  x,0\right)  =u_{0}\left(  x\right)&x\in\left(
0,L\right)  .
\end{cases}
\end{equation*}
We cite to the reader that the well-posedness theory for this system can be found in \cite{CaKawahara}. 

\subsection{Proof of Theorem \ref{main1}}We split the proof in two steps.  The first one provides an exact computation of the inner product $(L_1(u),L_2(u))$, whereas the second step gives the estimates obtained thanks to the pseudoconvexity conditions \eqref{ca2}.

\vspace{0.2cm}
\noindent\textbf{Step 1.} Exact computation of the scalar product $2(L_1(u),L_2(u))$.
\vspace{0.2cm}

First, let us compute the following
$$\int_{0}^{T}\int_{0}^{L}(Au+C_{1}u_{xx}+Eu_{xxxx})L_{2}(u)dxdt=:J_{1}+J_{2}+J_{3}$$
To do that, observe that $u$ belongs to $\mathcal{D}(P)$, thus, we infer by integrating by parts, that 
\begin{equation}\label{Au}
\begin{split}
J_{1}=& -\frac{1}{2}\int_{0}^{T}\int_{-L}^{L}[A_{t} - A_{5x} - (AC_{2})_{xx} + (AB)_{x} + (AD)_{xxx}]u^{2}dxdt\\
&-\frac{1}{2}\int_{0}^{T}\int_{-L}^{L}[5A_{xxx} - 3(AD)_{x} + 2(AC_{2})]u^{2}_{x}dxdt 
\\&+\frac{5}{2}\int_{0}^{T}\int_{-L}^{L}A_{x}u^{2}_{xx}dxdt,
\end{split}
\end{equation}
\begin{equation}\label{C111a}
\begin{split}
J_{2}=& \int_{0}^{T}\int_{-L}^{L}C_{1}u_{xx}[Bu_{x} + C_{2}u_{xx} + Du_{xxx}  - u_{xxxxx}]dxdt\\&+ \int_{0}^{T}\int_{-L}^{L}C_{1}u_{xx}u_{t}dxdt:=I_1+I_2.
\end{split}
\end{equation}
and
\begin{equation}\label{DFGH}
\begin{split}
J_{3}=& \int_{0}^{T}\int_{-L}^{L}Eu_{xxxx}[Bu_{x} + C_{2}u_{xx} + Du_{xxx} -u_{xxxxx}]dxdt\\
&+ \int_{0}^{T}\int_{-L}^{L}Eu_{xxxx}u_{t}dxdt:=I_3+I_4.
\end{split}
\end{equation}

Let us now treat $I_i$, for $i=1,2,3,4$.  Note that $I_1$ is equivalent to
\begin{equation}\label{C111}
\begin{split}
I_1=& -\frac{1}{2}\int_{0}^{T}\int_{-L}^{L}(C_{1}B)_{x}u^{2}_{x}dxdt -
\frac{1}{2}\int_{0}^{T}\int_{-L}^{L}[(C_{1}D)_{x} -2(C_{1}C_{2}) - C_{1xxx}]u_{xx}^{2}dxdt\\&- 
\frac{1}{2}\int_{0}^{T}\int_{-L}^{L}3C_{1x}u_{xxx}^{2}dxdt.
\end{split}
\end{equation}
By other hand, by the definition of $\omega$, see \eqref{omega ppp}, for $I_2$ we have that
\begin{equation}\label{C112}
\begin{split}
I_2
=&
-\frac{1}{2}\int_{0}^{T}\int_{-L}^{L}(AC_{1x})_{x}u^{2}dxdt - \int_{0}^{T}\int_{-L}^{L}(C_{1x})u_{xxx}^{2}dxdt\\
&-
\frac{1}{2}\int_{0}^{T}\int_{-L}^{L}[-2(BC_{1x}) + (CC_{1x})_{x} - (DC_{1x})_{xx}  \\
&+ 
(EC_{1x})_{xxx} - C_{1x} + (C_{1})_{xxxxx}]u_{x}^{2}dxdt\\
&-
\frac{1}{2}\int_{0}^{T}\int_{-L}^{L}[(2DC_{1x}) - 3(EC_{1x})_{x} - 4C_{1xxx}]u_{xx}^{2}dxdt \\
&-
\int_{0}^{T}\int_{-L}^{L}C_{1x}u_{x}\omega dxdt, 
\end{split}
\end{equation}
where we have used that $u$ belongs to $\mathcal{D}(P)$ and $u_{\mid t=0}=u_{\mid t=T}=0$.  Now, using the same strategy as before, that is, integration by parts,  $u$ belongs to $\mathcal{D}(P)$ and $u_{\mid t=0}=u_{\mid t=T}=0$ ensures that
\begin{equation}\label{RTYU}
\begin{split}
I_3=& -\frac{1}{2}\int_{0}^{T}\int_{-L}^{L}(EB)_{xxx}u^{2}_{x}dxdt+\frac{1}{2}\int_{0}^{T}\int_{-L}^{L}[3(EB)_{x} + (EC_{2})_{xx}]u_{xx}^{2}dxdt \\
&- \frac{1}{2}\int_{0}^{T}\int_{-L}^{L}[(ED)_{x} + 2(EC_{2})]u_{xxx}^{2}dxdt + \frac{1}{2}\int_{0}^{T}\int_{-L}^{L}E_{x}u^{2}_{xxxx}dxdt,
\end{split}
\end{equation}
and
\begin{equation}\label{3W2A}
\begin{split}
I_4
=&
\int_{0}^{T}\int_{-L}^{L}E_{xx}u_{t}u_{xx}dxdt - 2\int_{0}^{T}\int_{-L}^{L}(E_{x}u_{xx})_{x}u_{t}dxdt +\frac{1}{2}\int_{0}^{T}\int_{-L}^{L}E\frac{d}{dt}u^{2}_{xx}dxdt\\
=&
-\int_{0}^{T}\int_{-L}^{L}E_{xx}u_{xx}u_{t}dxdt - 2\int_{0}^{T}\int_{-L}^{L}E_{x}u_{xxx}u_{t}dxdt - \frac{1}{2}\int_{0}^{T}\int_{-L}^{L}E_{t}u^{2}_{xx}dxdt\\
=&
-\int_{0}^{T}\int_{-L}^{L}[E_{xx}u_{xx} + 2E_{x}u_{xxx}]u_{t}dxdt - \frac{1}{2}\int_{0}^{T}\int_{-L}^{L}E_{t}u^{2}_{xx}dxdt=:I_5+I_6.
\end{split}
\end{equation}
Note that $I_5$ can be seen as 
\begin{equation*}
\begin{split}
I_5=&
\frac{1}{2}\int_{0}^{T}\int_{-L}^{L}[(E_{xx}A)_{xx} - 2(E_{x}A)_{xxx}]u^{2}dxdt \\
&+ \frac{1}{2}\int_{0}^{T}\int_{-L}^{L}\left[-2(E_{xx}A) - (BE_{xx})_{x}+6(E_{x}A)_{x} + 2(E_{x}B)_{xx}\right]u^{2}_{x}dxdt\\
&+
\frac{1}{2}\int_{0}^{T}\int_{-L}^{L}[ 2(E_{xx}C) - (E_{xx}D)_{x} + (E_{xx}E)_{xx} +E_{xxxxx} - 4(E_{x}B) - 2(CE_{x})_{x}]u^{2}_{xx}dxdt\\
&+
\frac{1}{2}\int_{0}^{T}\int_{-L}^{L}[ - 2(E_{xxx}E) - 7E_{xxx} +4(E_{x}D) - 2(EE_{x})_{x}]u^{2}_{xxx}dxdt \\
&+
\int_{0}^{T}\int_{-L}^{L}E_{x}u^{2}_{xxxx}dxdt-
\int_{0}^{T}\int_{-L}^{L}(E_{xx}u_{xx} + 2E_{x}u_{xxx})\omega dxdt,
\end{split}
\end{equation*}
thanks to  \eqref{omega ppp}.  So,  putting the previous equality into \eqref{3W2A} we get, 
\begin{equation}\label{BNM}
\begin{split}
&I_4=\frac{1}{2}\int_{0}^{T}\int_{-L}^{L}[(E_{xx}A)_{xx} - 2(E_{x}A)_{xxx}]u^{2}dxdt \\
&+ \frac{1}{2}\int_{0}^{T}\int_{-L}^{L}\left[-2(E_{xx}A) - (BE_{xx})_{x}+(E_{x}A)_{x} + 2(E_{x}B)_{xx}\right]u^{2}_{x}dxdt\\
&+
\frac{1}{2}\int_{0}^{T}\int_{-L}^{L}[ 2(E_{xx}C) - (E_{xx}D)_{x} + (E_{xx}E)_{xx} +E_{xxxxx} - 4(E_{x}B) -E_{t}- 2(CE_{x})_{x}]u^{2}_{xx}dxdt\\
&+
\frac{1}{2}\int_{0}^{T}\int_{-L}^{L}[ - 2(E_{xxx}E) - 7E_{xxx} +4(E_{x}D) - 2(EE_{x})_{x}]u^{2}_{xxx}dxdt \\
&+
\int_{0}^{T}\int_{-L}^{L}E_{x}u^{2}_{xxxx}dxdt-
\int_{0}^{T}\int_{-L}^{L}(E_{xx}u_{xx} + 2E_{x}u_{xxx})\omega dxdt.
\end{split}
\end{equation}
Putting together  \eqref{C111} and \eqref{C112} in \eqref{C111a},  \eqref{RTYU} and \eqref{BNM} into \eqref{DFGH}, and adding the result quantities with \eqref{Au},  we have that the scalar product $2(L_1(u),L_2(u))$ is given by 
\begin{equation}\label{MNOT}
\begin{split}
2\int_{0}^{T}\int_{-L}^{L}L_{1}(u)L_{2}(u)dxdt =& -
\int_{0}^{T}\int_{-L}^{L}(E_{xx}u_{xx} + 2E_{x}u_{xxx})\omega dxdt\\
&-2\int_{0}^{T}\int_{-L}^{L}(\omega C_{1x})u_{x}dxdt+\int_{0}^{T}\int_{-L}^{L}Mu^{2}dxdt\\
&+
\int_{0}^{T}\int_{-L}^{L}Nu^{2}_{x}dxdt + \int_{0}^{T}\int_{-L}^{L}Ou^{2}_{xx}dxdt\\
&+
\int_{0}^{T}\int_{-L}^{L}Ru^{2}_{xxx}dxdt + \int_{0}^{T}\int_{-L}^{L}Su^{2}_{4x}dxdt,
\end{split}
\end{equation}
where
\begin{equation*}
\begin{split}
M=&- (AB)_{x} - A_{t} + A_{5x} + (AC_{2})_{xx} - (AD)_{xxx} -(AC_{1x})_{x} + (E_{xx}A)_{xx} - 2(E_{x}A)_{xxx}\\
N=& \ 3(AD)_{x} -2(AC_{2}) - (C_{1}B)_{x} + (BC_{1x}) + C_{1x} - (CC_{1x})_{x} + (DC_{1x})_{xx} -5A_{xxx}  \\
&-
(EC_{1x})_{xxx} - C_{15x} -(EB)_{xxx} -2(E_{xx}A) - (BE_{xx})_{x} + 6(E_{x}A)_{x} + 2(E_{x}B)_{xx}\\
O=&\ 5A_{x} -(C_{1}D)_{x} - 2(DC_{1x}) + 3(EB)_{x} + 2(C_{1}C_{2}) - 4(E_{x}B) + 5C_{1xxx} + 3(EC_{1x})_{x}\\
&+
 2(E_{xx}C) + (EC_{2})_{xx} - (E_{xx}D)_{x} + (E_{xx}E)_{xx} + E_{5x} - E_{t} - 2(CE_{x})_{x}\\
R=&- 5C_{1x} -(ED)_{x} + 4(E_{x}D) - 2(EC_{2}) - 2(E_{xxx}E) - 7E_{xxx} - 2(EE_{x})_{x}\\
S=&\ 3E_{x}
\end{split}
\end{equation*}

Now, note that 
\begin{equation*}
\begin{split}
2 \int_{0}^{T} \int_{-L}^{L} L_{1}(u) L_{2}(u) d x d t & \leq \int_{0}^{T} \int_{-L}^{L}\left(L_{1}(u)+L_{2}(u)\right)^{2} d x d t  \leq \int_{0}^{T} \int_{-L}^{L} \omega^{2} d x d t,
\end{split}
\end{equation*}
we have due to \eqref{MNOT} that
\begin{equation}\label{12}
\begin{split}
&\int_{0}^{T}\int_{-L}^{L}Mu^{2}dxdt
+
\int_{0}^{T}\int_{-L}^{L}Nu^{2}_{x}dxdt + \int_{0}^{T}\int_{-L}^{L}Ou^{2}_{xx}dxdt+
\int_{0}^{T}\int_{-L}^{L}Ru^{2}_{xxx}dxdt \\&+ \int_{0}^{T}\int_{-L}^{L}Su^{2}_{xxxx}dxdt-
2\int_{0}^{T}\int_{-L}^{L}(\omega C_{1x})u_{x}dxdt - \int_{0}^{T}\int_{-L}^{L}(E_{xx}u_{xx} + 2E_{x}u_{xxx})\omega dxdt\\
&\leq
\int_{0}^{T}\int_{-L}^{L}\omega^{2}dxdt.
\end{split}
\end{equation}
Let us put each common term of the previous inequality together. To do that, note that using Young inequality,  for  $\epsilon \in (0,1)$ we get
\begin{equation*}
\begin{split}
2\int_{0}^{T}\int_{-L}^{L}(\omega C_{1x})u_{x}dxdt=& \ 2\int_{0}^{T}\int_{-L}^{L}\left(\epsilon^{\frac{1}{2}} C_{1x}u_{x}\right)\left(\epsilon^{-\frac{1}{2}}\omega\right)dxdt\\
\leq& 
\ \epsilon\int_{0}^{T}\int_{-L}^{L}C^{2}_{1x}u^{2}_{x}dxdt + \epsilon^{-1}\int_{0}^{T}\int_{-L}^{L}\omega^{2}dxdt.
\end{split}
\end{equation*}
In an analogous way,
\begin{equation*}
\begin{split} 
\int_{0}^{T}\int_{-L}^{L}(E_{xx}u_{xx} + 2E_{x}u_{xxx})\omega dxdt \leq&\ \frac{\epsilon}{2}\int_{0}^{T}\int_{-L}^{L}E_{xx}^{2}u^{2}_{xx}dxdt + \epsilon\int_{0}^{T}\int_{-L}^{L}E_{x}^{2}u^{2}_{xxx}dxdt\\
&+
\frac{3}{2}\epsilon^{-1}\int_{0}^{T}\int_{-L}^{L}\omega^{2}dxdt.
\end{split}
\end{equation*}
So, we have that
\begin{equation}\label{Q}
\begin{split} 
-\epsilon\int_{0}^{T}\int_{-L}^{L}C^{2}_{1x}u^{2}_{x}dxdt - \epsilon^{-1}\int_{0}^{T}\int_{-L}^{L}\omega^{2}dxdt\leq 
-2\int_{0}^{T}\int_{-L}^{L}(\omega C_{1x})u_{x}dxdt
\end{split}
\end{equation}
and
\begin{equation}\label{W}
\begin{split} 
-\frac{\epsilon}{2}\int_{0}^{T}\int_{-L}^{L}E_{xx}^{2}u^{2}_{xx}dxdt &- \epsilon\int_{0}^{T}\int_{-L}^{L}E_{x}^{2}u^{2}_{xxx}dxdt - \frac{3}{2}\epsilon^{-1}\int_{0}^{T}\int_{-L}^{L}\omega^{2}dxdt\\\leq& 
-\int_{0}^{T}\int_{-L}^{L}(E_{xx}u_{xx} + 2E_{x}u_{xxx})\omega dxdt.
\end{split}
\end{equation}
Replacing \eqref{Q} and \eqref{W} into \eqref{12} yields that
\begin{equation}\label{semifinal}
\begin{split} 
&\int_{0}^{T}\int_{-L}^{L}M u^{2}dxdt
+
\int_{0}^{T}\int_{-L}^{L}\left(N - \epsilon C^{2}_{1x}\right)u^{2}_{x}dxdt + \int_{0}^{T}\int_{-L}^{L}\left(O - \frac{\epsilon}{2} E_{xx}^{2}\right)u^{2}_{xx}dxdt\\
&+
\int_{0}^{T}\int_{-L}^{L}\left(R - \epsilon E^{2}_{x}\right)u^{2}_{xxx}dxdt +\int_{0}^{T}\int_{-L}^{L}S u^{2}_{xxxx}dxdt
\leq
\left(1 + \frac{5}{2}\epsilon^{-1}\right)\int_{0}^{T}\int_{-L}^{L}\omega^{2}dxdt.
\end{split}
\end{equation}

\vspace{0.2cm}
\noindent\textbf{Step 2.} Estimation of each term of the left hand side of \eqref{semifinal}.
\vspace{0.2cm}

The estimates are given in a series of claims.

\vspace{0.2cm}
\noindent\textbf{Claim 1.} There exist some constants $s_1 > 0$ and $C_1 >1$ such that for all $s\geq s_1$, we have
\vspace{0.2cm}
$$ \int_{0}^{T}\int_{-L}^{L}Mu^{2}dxdt\geq C_1^{-1}\int_{0}^{T}\int_{-L}^{L}(s\varphi)^9u^2dxdt.$$

Observe that 
\begin{equation*}
\begin{split}
M =& -(AB)_{x}+\frac{O\left(s^{8}\right)}{t^{8}(T-t)^{8}}
= - 45s^{9}\varphi_{x}^{8}\varphi_{xx}+\frac{O\left(s^{8}\right)}{t^{8}(T-t)^{8}} = - 45s^{9}\frac{(\psi')^8\psi''}{{t^{9}(T-t)^{9}}}+\frac{O\left(s^{8}\right)}{t^{8}(T-t)^{8}}
\end{split}
\end{equation*}
We infer from \eqref{ca2} that for some $k_1 > 0$ and all $ s>0$, large enough,  we have 
\begin{equation*}
M \geq k_1\frac{s^9}{t^9(T-t)^9}
\end{equation*}
Claim 1 follows then for all $s>s_1$,  with $s_1$ large enough and some $C_1>1$.

\vspace{0.2cm}
\noindent\textbf{Claim 2.} There exist some constants $s_2 > 0$ and $C_2 >1$ such that for all $s\geq s_2$, we have
\vspace{0.2cm}
$$\int_{0}^{T}\int_{-L}^{L}\left(N - \epsilon C^{2}_{1x}\right)u^{2}_{x}dxdt \geq C_2^{-1}\int_{0}^{T}\int_{-L}^{L}(s\varphi)^7u_x^2dxdt.$$

Noting that 
\begin{equation*}
\begin{split}
N - \epsilon C^{2}_{1x}=& 3(AD)_{x} -2(AC_{2}) - (C_{1}B)_{x} + (BC_{1x})
+\frac{O\left(s^{6}\right)}{t^{6}(T-t)^{6}} \\
=&-50s^{7}\varphi^{6}_{x}\varphi_{xx} +\frac{O\left(s^{6}\right)}{t^{6}(T-t)^{6}} = -50s^{7}\frac{(\psi')^6\psi''}{{t^{7}(T-t)^{7}}}+\frac{O\left(s^{6}\right)}{t^{6}(T-t)^{6}},
\end{split}
\end{equation*}
and using again that \eqref{ca2} holds, we get for some $k_2 > 0$ and all $ s>0$, large enough,  that 
\begin{equation*}
N - \epsilon C^{2}_{1x}\geq k_2\frac{s^7}{t^7(T-t)^7}
\end{equation*}
and Claim 2 follows then for all $s>s_2$,  with $s_2$ large enough and some $C_2>1$.

\vspace{0.2cm}
\noindent\textbf{Claim 3.} There exist some constants $s_3 > 0$ and $C_3 >1$ such that for all $s\geq s_3$, we have
\vspace{0.2cm}
$$\int_{0}^{T}\int_{-L}^{L}\left(O - \frac{\epsilon}{2} E_{xx}^{2}\right)u^{2}_{xx}dxdt\geq C_3^{-1}\int_{0}^{T}\int_{-L}^{L}(s\varphi)^5u_{xx}^2dxdt.$$

First, see that
\begin{equation*}
\begin{split}
O - \frac{\epsilon}{2} E_{xx}^{2}=&\ 5A_{x} -(C_{1}D)_{x} - 2(DC_{1x}) + 3(EB)_{x} + 2(C_{1}C_{2}) - 4(E_{x}B)+\frac{O\left(s^{4}\right)}{t^{4}(T-t)^{4}} \\
= &-250s^{5}\varphi_{x}^{4}\varphi_{xx} +\frac{O\left(s^{4}\right)}{t^{4}(T-t)^{4}} = -250s^{5}\frac{(\psi')^4\psi''}{{t^{5}(T-t)^{5}}} +\frac{O\left(s^{4}\right)}{t^{4}(T-t)^{4}}.
\end{split}
\end{equation*}
Next,  using \eqref{ca2} we have that for some $k_3> 0$ and all $ s>0$, large enough, 
\begin{equation*}
O - \frac{\epsilon}{2} E_{xx}^{2} \geq k_{3}\frac{s^{5}}{t^{5}(T - t)^{5}}
\end{equation*}
is verified,  so Claim 3 holds true for all $s>s_3$,  with $s_3$ large enough and some $C_3>1$.

\vspace{0.2cm}
\noindent\textbf{Claim 4.} There exist some constants $s_4 > 0$ and $C_4 >1$ such that for all $s\geq s_4$, we have
\vspace{0.2cm}
$$\int_{0}^{T}\int_{-L}^{L}\left(R - \epsilon E^{2}_{x}\right)u^{2}_{xxx}dxdt\geq C_4^{-1}\int_{0}^{T}\int_{-L}^{L}(s\varphi)^3u^2_{xxx}dxdt.$$

As the previous Claims,  thanks to  \eqref{ca2} and 
\begin{equation*}
\begin{split}
R - \epsilon E^{2}_{x}=& -5C_{1x} -(ED)_{x} + 4(E_{x}D) - 2(EC_{2}) + \frac{O\left(s^{2}\right)}{t^{2}(T-t)^{2}}\\
=&-100 s^{3}\varphi_{x}^{2}\varphi_{xx} + \frac{O\left(s^{2}\right)}{t^{2}(T-t)^{2}}= -100s^{3}\frac{(\psi')^2\psi''}{{t^{3}(T-t)^{3}}} + \frac{O\left(s^{2}\right)}{t^{2}(T-t)^{2}},
\end{split}
\end{equation*}
we can find some constant $k_4> 0$ and all $ s>0$, large enough, such that 
\begin{equation*}
R - \epsilon E^{2}_{x} \geq k_{4}\frac{s^{3}}{t^{3}(T - t)^{3}}
\end{equation*}
follows and Claim 4 is verified  for all $s>s_4$,  with $s_4$ large enough and some $C_4>1$.

\vspace{0.2cm}
\noindent\textbf{Claim 5.} There exist some constants $s_5 > 0$ and $C_5 >1$ such that for all $s\geq s_4$, we have
\vspace{0.2cm}
$$\int_{0}^{T}\int_{-L}^{L}S u^{2}_{xxxx}dxdt\geq C_5^{-1}\int_{0}^{T}\int_{-L}^{L}(s\varphi)u^2_{xxxx}dxdt.$$

This is also a direct consequence of the fact that $ S = - s5\varphi_{xx}$ and  \eqref{ca2} holds. Therefore, Claim 5 is verified.

\vspace{0.2cm}

We infer from Steps 1 and 2,  that for some positive constants $s_0$, $C$, and all $s\geq s_0$, we have 
\begin{equation*}
\begin{split}
\int_{0}^{T}\int_{-L}^{L}\left\{(s\varphi)^9|u|^2+(s\varphi)^{7}|u_x|^2 + (s\varphi)^{5}|u_{xx}|^{2}+(s\varphi)^{3}|u_{xxx}|^{2}+ s\varphi|u_{xxxx}|^{2}\right\}dxdt \\ \leq C\int_{0}^{T}\int_{-L}^{L}|\omega|^{2}dxdt.
\end{split}
\end{equation*}
Replacing $u$ by $e^{-s\varpi}q$ yields \eqref{DGC}. \qed

\section{Approximation Theorem}\label{Sec3}
This section is devoted to presenting an application of the Carleman estimate shown in Section \ref{Sec2} for the Kawahara operator $P$ defined by \eqref{P}-\eqref{D_P}.  First, we prove a result which is the key to proving the approximation Theorem \ref{main2}.  We have the following as a consequence of the Theorem \ref{main1}. 

\begin{proposition}\label{3.2} For $L > 0$ and $f= f(t,x)$ a function in $L^{2}({\mathbb{R}}\times{(-L,L)})$ with $\supp \ f \subset ([t_{1}, t_{2}]\times(-L,L))$, where $-\infty < t_{1} < t_{2} < \infty,$ we have that for every $\epsilon > 0$ there exist a positive number  $C= C(L, t_{1}, t_{2}, \epsilon)$ ($C$ does not depend on $f$) and a function $v \in L^{2}({\mathbb{R}}\times{(-L, L)})$ such that
\begin{equation*}
\begin{cases}
v_{t} +  v_{x} +  v_{xxx} - v_{xxxxx}= f \ \mbox{in} \ \mathcal{D}'({\mathbb{R}}\times{(-L,L)}),\\
\supp \ v \subset [t_{1} - \epsilon, t_{2} - \epsilon] \times (-L,L)
\end{cases}
\end{equation*}
and
\begin{equation*}
\|v\|_{L^{2}({\mathbb{R}}\times{(-L, L)})} \leq C\|f\|_{L^{2}({\mathbb{R}}\times{(-L, L)})}.
\end{equation*}
\end{proposition}
\begin{proof}
By a change of variable, if necessary, and without loss of generality, we may assume that  $0= t_{1} - \epsilon < t_{1}< t_{2} < t_{2} - \epsilon = T$.  Thanks to the Calerman estimate \eqref{DGC},  we have that 
\begin{equation}\label{1.2}
\int_{0}^{T}\int_{-L}^{L}|q|^{2}e^{-\frac{k}{t(T - t)}}dxdt \leq C_{1}\int_{0}^{T}\int_{-L}^{L}|P(q)|^{2}dxdt,
\end{equation}
for some $k>0$,  $C_{1} > 0$ and any  $q \in \mathcal{Z}$. Here, the operator $P$ is defined by \eqref{P}. Therefore, we have that $ F: \mathcal{Z}\times \mathcal{Z} \longrightarrow \mathbb{R}$ defined by $$F(p, q)= \int_{0}^{T}\int_{-L}^{L}P(p)P(q)dxdt$$
is a scalar product in $\mathcal{Z}$.  Now, let us consider $H$ the completion of $\mathcal{Z}$ for $(\cdot,\cdot)$.  Note that $|q|^{2}e^{-\frac{k}{t(T - t)}}$ is integrable on $Q_{T}$ if $q\in H$ and \eqref{1.2} holds true.  By the other hand, we claim that $T: H \longrightarrow \mathbb{R}$ defined by
$$ T(q)= -\int_{0}^{T}\int_{-L}^{L}f(t,x)q(x)dxdt,$$
is well-defined on $H$. In fact,  due the hypotheses, that is, $\supp \ f \subset ([t_{1}, t_{2}]\times(-L,L))$, and thanks to Hölder inequality and the relation \eqref{1.2}, we have 
\begin{equation}\label{1.2.2}
 \int_{0}^{T}\int_{-L}^{L}|f(t,x)q(x)|dxdt \leq \int_{t_{1}}^{t_{2}}\int_{-L}^{L}|f(t,x)q(x)|dxdt\leq C\|f(t,x)\|_{L^{2}((t_{1},t_{2})\times(-L,L))}(q,q)^{\frac{1}{2}},
\end{equation}
for some constant positive $C$. 

Thus, it follows from the Riesz representation theorem that there exists a unique $u \in H$ such that
\begin{equation}\label{1.3}
F(u,q)= T(q), \ \forall q \in H.
\end{equation}
Pick $v:= P(u) \in L^{2}((0,T)\times(-L,L))$, so have that 
\begin{equation*}
\begin{split}
\langle P^{*}(v), q\rangle=& \langle v, P(q)\rangle= 
\int_{0}^{T}\int_{-L}^{L}vP(q)dxdt= 
\int_{0}^{T}\int_{-L}^{L}P(u)P(q)dxdt\\
=&
F(u,q)=T(q)=
-\int_{0}^{T}\int_{-L}^{L}fq dxdt= \langle-f, q\rangle,
\end{split}
\end{equation*}
where $\langle\cdot,\cdot\rangle$ denotes the duality pairing $\langle\cdot,\cdot\rangle_{\mathcal{D}'(Q_{T});\mathcal{D}(Q_{T})}$ and $P^{*}= -P$,  hence $$Pv= f \ \text{in} \ \mathcal{D}'(Q_{T}).$$

Finally, observe that $v \in H^{1}((0,T); H^{-5}(-L,L))$,  since we have $$v_t=f+v_{xxxxx}- v_{xxx}- v_x \in L^2(0,T;H^{-5}(-L,L)),$$ thus $v(0,\cdot)$ and $v(T,\cdot)$ make sense in $H^{-5}(-L,L)$.  Now, let $q \in H^{1}(0,T; H^{5}_{0}(-L,L))$, follows by \eqref{1.3} that
\begin{equation*}
\begin{split}
-\int_{0}^{T}\int_{-L}^{L}fq dxdt=
-\int_{0}^{T}\int_{-L}^{L}fq dxdt t + \langle v(t,x), q(t,x) \rangle\bigg\vert_{t=0}^{T},
\end{split}
\end{equation*}
where $\langle\cdot,\cdot\rangle$ denotes the duality pairing $\langle\cdot,\cdot\rangle_{H^{-5}(-L,L);H^5_0(-L,L)}$.  Since $q\vert_{t=0}$ and $q\vert_{t=T}$ are arbitrarily in $\mathcal{D}(-L,L)$, we infer that $v(T,\cdot)= v(0,\cdot)= 0$ in $H^{-5}(-L,L)$.  Therefore, the result follows extending $v$ by setting $v(t,x)=0$ for $(t,x) \notin Q_{T}$. 
\end{proof}

Now, we are in a position to prove Theorem \ref{main2}.  
 
\begin{proof}[Proof of Theorem \ref{main2}]
Pick $\eta > 0$, to be chosen later.  Thanks to the Lemma \ref{4.2},  applied for $L= n + 1, \ l_{1}= n-1, \ l_{2}= n, \ 2\delta= \frac{\epsilon}{2}$,  there exists $\tilde{v} \in L^{2}((0,T)\times(-n-1,n + 1))$ such that 
\begin{equation*}
P\tilde{v}= 0 \ \mbox{in} \ (0,T)\times(-n -1,n + 1).
\end{equation*}
\begin{equation}\label{4.20}
\tilde{v}(t,.)= S_{n + 1}(t - t_{1} + \frac{\epsilon}{2})v_{1}, \ \mbox{for} \ t_{1} - \frac{\epsilon}{2} < t < t_{1} -  \frac{\epsilon}{4}
\end{equation}
and
\begin{equation}\label{4.21}
\tilde{v}(t,.)= S_{n + 1}(t - t_{2} -  \frac{\epsilon}{4})v_{2}, \ \mbox{for} \ t_{2} +  \frac{\epsilon}{4} < t < t_{2} +  \frac{\epsilon}{2},
\end{equation}
for some $(v_{1}, v_{2}) \in L^{2}((t_{1} - \frac{\epsilon}{2}, t_{2} + \frac{\epsilon}{2})\times(-n + 1, n - 1))^{2}$ and
\begin{equation*}
\|\tilde{v} - u\|_{L^{2}((t_{1} - \frac{\epsilon}{2}, t_{2} + \frac{\epsilon}{2})\times(-n + 1, n - 1))} < \eta.
\end{equation*}
So that \eqref{4.18} be fulfilled,  we multiply $\tilde{v}$ by a cut-off function. Now on, consider $\varphi \in \mathcal{D}(0,T)$ be such that $0\leq \varphi \leq 1, \ \varphi(t)=1,$ for all $t\in[t_{1} - \frac{\epsilon}{4}, t_{2} + \frac{\epsilon}{4}]$ and $\supp \ \varphi \subset [t_{1} - \frac{\epsilon}{2}, t_{2} + \frac{\epsilon}{2}].$ Picking $\overline{v}(t,x)= \varphi(t)\tilde{v}(t,x), $ we get
\begin{equation*}
\supp  \ \overline{v} \subset [t_{1} - \frac{\epsilon}{2}, t_{2} + \frac{\epsilon}{2}]\times(-n - 1, n + 1).
\end{equation*}  
Therefore, 
\begin{equation*}
\begin{split}
\|\overline{v} - u\|_{L^{2}((0,T)\times(-n + 1, n - 1))} \leq& \|\tilde{v} - u\|_{L^{2}((t_{1} - \frac{\epsilon}{2}, t_{2} + \frac{\epsilon}{2})\times(-n + 1, n - 1))} \\
&+ 
\|(\varphi - 1)\tilde{v}\|_{L^{2}((t_{1} - \frac{\epsilon}{2}, t_{2} + \frac{\epsilon}{2})\times(-n + 1, n - 1))}.
\end{split}
\end{equation*}
Since $\supp \ u \subset [t_{1}, t_{2}]\times(-n,n)$ and $\varphi(t)= 1,$ for $t_{1} - \frac{\epsilon}{4} \leq t \leq t_{2} + \frac{\epsilon}{4}$,  we have
\begin{equation}\label{4.22}
\begin{split}
\|(\varphi - 1)\tilde{v}\|^{2}_{L^{2}((t_{1} - \frac{\epsilon}{2}, t_{2} + \frac{\epsilon}{2})\times(-n + 1, n - 1))}
\leq&
\|\tilde{v}\|^{2}_{L^{2}(\{(t_{1} - \frac{\epsilon}{2}, t_{1} - \frac{\epsilon}{4})\cup(t_{2} + \frac{\epsilon}{4}, t_{2} + \frac{\epsilon}{2})\}\times(-n + 1, n - 1))}
\\=&
\|\tilde{v} - u\|^{2}_{L^{2}(\{(t_{1} - \frac{\epsilon}{2}, t_{1} - \frac{\epsilon}{4})\cup(t_{2} + \frac{\epsilon}{4}, t_{2} + \frac{\epsilon}{2})\}\times(-n + 1, n - 1))} 
\\
\leq&
\|\tilde{v} - u\|^{2}_{L^{2}((t_{1} - \frac{\epsilon}{2}, t_{2} + \frac{\epsilon}{2})\times(-n + 1, n - 1))} \\
\leq&
\eta^{2}.
\end{split}
\end{equation}
Hence, 
\begin{equation}\label{4.23}
\|\overline{v} - u\|_{L^{2}((0,T)\times(-n + 1, n - 1))} \leq 2\eta,
\end{equation}
where we have used the fact that $\supp \ u \subset [t_{1}, t_{2}]\times(-n,n)$.  Finally, $$P\overline{v}= \frac{d\varphi}{dt}\tilde{v}\quad \text{in} \quad (0,T)\times(-n - 1, n + 1)$$ so
\begin{equation*}
\|P\overline{v}\|^{2}_{L^{2}((0,T)\times(-n -1, n + 1))} \leq \bigl\|\frac{d\varphi}{dt}\bigr\|^{2}_{L^{\infty}(0,T)}\|\tilde{v}\|^{2}_{L^{2}(\{(t_{1} - \frac{\epsilon}{2}, t_{1} - \frac{\epsilon}{4})\cup(t_{2} + \frac{\epsilon}{4}, t_{2} + \frac{\epsilon}{2})\}\times(-n - 1, n + 1))} 
\end{equation*}
thanks to the fact that $\varphi(t)= 1$ in  $[t_{1} - \frac{\epsilon}{4}, t_{1} + \frac{\epsilon}{4}]$. 
On the other hand,  since \eqref{4.20} and  \eqref{4.21} holds,  we infer by the observability result, that is,  by Lemma \ref{OI_I},  that there exists a constant $C=C(n,\epsilon)>0$ such that 
\begin{equation*}
\|\tilde{v}\|_{L^{2}((t_{1} - \frac{\epsilon}{2}, t_{1} - \frac{\epsilon}{4})\times(-n - 1, n + 1))} \leq C\|\tilde{v}\|_{L^{2}((t_{1} - \frac{\epsilon}{2}, t_{1} - \frac{\epsilon}{4})\times(-n + 1, n - 1))}
\end{equation*}
and also
\begin{equation*}
\|\tilde{v}\|_{L^{2}((t_{2} + \frac{\epsilon}{4}, t_{2} + \frac{\epsilon}{2})\times(-n - 1, n + 1))} \leq C\|\tilde{v}\|_{L^{2}((t_{2} + \frac{\epsilon}{4}, t_{1} + \frac{\epsilon}{2})\times(-n + 1, n - 1))},
\end{equation*}
or equivalently, 
\begin{equation*}
\|\tilde{v}\|_{L^{2}(\{(t_{1} - \frac{\epsilon}{2}, t_{1} - \frac{\epsilon}{4})\cup(t_{2} + \frac{\epsilon}{4}, t_{2} + \frac{\epsilon}{2})\}\times(-n - 1, n + 1))} \leq C \|\tilde{v}\|_{L^{2}(\{(t_{1} - \frac{\epsilon}{2}, t_{1} - \frac{\epsilon}{4})\cup(t_{2} + \frac{\epsilon}{4}, t_{2} + \frac{\epsilon}{2}))\times(-n + 1, n - 1))}.
\end{equation*}
Thus, combining the last inequality with  \eqref{4.22} yields that
\begin{equation}\label{4.24}
\|P\overline{v}\|_{L^{2}((0,T)\times(-n - 1, n + 1))} \leq C \bigl\|\frac{d\varphi}{dt}\bigr\|_{L^{\infty}(0,T)}\eta
\end{equation}

Now, to finish the proof, we use Proposition \ref{3.2},  to ensure the existence of a constant $C = C'(n, t_{1}, t_{2}, \epsilon) > 0$ and a function $\omega \in L^{2}((0,T)\times(-n - 1, n + 1))$ such that
\begin{equation}\label{asterisco}
\begin{cases}
P\omega= P\overline{v} \ \mbox{in} \ (0, T)\times(-n - 1, n + 1), \\ 
\supp \ \omega \subset [t_{1} - \epsilon, t_{2} + \epsilon]\times(-n - 1, n + 1),
\end{cases}
\end{equation} 
and
\begin{equation}\label{4.25}
\|\omega\|_{L^{2}((0,T)\times(-n - 1, n + 1))} \leq C'\|P\overline{v}\|_{L^{2}((0,T)\times(-n - 1, n + 1))}.
\end{equation}
Consequently, setting $v = \overline{v} - \omega$ we get \eqref{4.17} and \eqref{4.18} by using  \eqref{asterisco}. Moreover,  thanks to  \eqref{4.23}, \eqref{4.24} and \eqref{4.25}, we get that
\begin{equation*}
\|v - u\|_{L^{2}((0,T)\times(-n + 1, n - 1))} \leq \bigl(2 + CC'\bigl\|\frac{d\varphi}{dt}\bigr\|_{L^{\infty}(0,T)}\bigr)\eta.
\end{equation*}
Now, choosing  $\eta$ small enough, we have shown \eqref{4.19} and so the result is shown.
\end{proof}

Finally, as a consequence of Theorem \ref{main2}, we prove the next result that gives us information to prove the third main result of the article in the next section.

\begin{corollary}\label{Prop princ} Let $t_{1}$, $t_{2}$, $T$ real numbers such thar  $0 < t_{1} < t_{2} < T$ and $f= f(t,x)$ be a function in $L^{2}_{loc}(\mathbb{R}^{2})$ such that
\begin{equation*}
\supp \ f \subset {[t_{1}, t_{2}]}\times{\mathbb{R}}.
\end{equation*}
Let $\epsilon \in (0, min(t_{1},T - t_{2}))$,  then there exists $u \in L^{2}_{loc}(\mathbb{R}^{2})$ such that 
\begin{equation*}
\omega_{t} +  \omega_{x} +  \omega_{xxx} - \omega_{xxxxx}= f \ \mbox{in} \  \mathcal{D}'(\mathbb{R}^{2})
\end{equation*}
and
\begin{equation*}
\supp \ \omega \subset [t_{1}-\epsilon, t_{2} + \epsilon]\times{\mathbb{R}}.
\end{equation*}
\end{corollary}
\begin{proof}
Consider two sequences of number denoted by $\{t^{n}_{1}\}_{n \geq 2}$ and $\{t^{n}_{2}\}_{n \geq 2}$ such that for all $n\geq 2$ we have
\begin{equation}\label{4.26}
t_{1} - \epsilon < t^{n + 1}_{1} < t^{n}_{1} < t_{1} < t_{2} < t^{n}_{2} < t^{n + 1}_{2} < t_{2} + \epsilon.
\end{equation}
We construct by induction over $n$ a sequence  $\{u_n\}_{n\geq2}$ of function  such that, for every $n\geq 2$
\begin{equation}\label{R}
\begin{cases}
u_{n} \in L^{2}((0,T)\times(-n,n)),\\
\supp \ u_{n} \subset [t^{n}_{1}, t^{n}_{2}]\times(-n,n),\\
Pu_{n}= f \ \mbox{in} \ (0,T)\times(-n,n),
\end{cases}
\end{equation}
and, if $n>2$
\begin{equation}\label{R1}
\|\tilde{u}_{n} - u_{n-1}\|_{L^{2}((0,T)\times(-n+2,n-2))} < \frac{1}{2^{n}}.
\end{equation}
Here,  $u_2$ is given by Proposition \ref{3.2}.  Now on, let us assume, for $n\geq 2$, that $u_2,\cdots,u_n$ satisfies \eqref{R} and \eqref{R1}.  By  Proposition \ref{3.2}, there exists  $\omega \in L^{2}((0,T)\times(-n-1,n+1))$ such that
$$\supp \ \omega \subset [t^{2}_{1}, t^{2}_{2}]\times(-n-1,n+1)$$
and 
$$  P\omega= f \ \mbox{in} \ (0,T)\times(-n-1,n+1).$$
As we have $P(u_{n}-\omega)= 0$ in $(0,T)\times(-n,n)$ and
\begin{equation*}
\supp \ (u_{n}- \omega) \subset [t^{n}_{1}, t^{n}_{2}]\times(-n,n)
\end{equation*}
with $t^{n+1}_{1} < t^{n}_{1}< t^{n}_{2} < t^{n+1}_{2}$. So, using Theorem \ref{main2}, there exists a function $v \in L^{2}((0,T)\times(-n-1,n+1))$ such that
\begin{equation*}
\supp \ v \subset [t^{n+1}_{1}, t^{n+1}_{2}]\times(-n-1,n+1), \  \ Pv= 0 \ \mbox{in} \ (0,T)\times(-n-1,n+1)
\end{equation*}
and
\begin{equation*}
\|v - (u_{n} - \omega)\|_{L^{2}((0,T)\times(-n+1,n-1))} < \frac{1}{2^{n-1}}.
\end{equation*}
Thus, picking $u_{n+1}= v + \omega$, we get that  $u_{n+1}$ satisfies \eqref{R} and \eqref{R}.  Extending the sequence $\{u_{n}\}_{n\geq2}$ by $u_{n}(t,x)=0$ for $(t,x) \in \mathbb{R}^{2}\setminus (0,T)\times(-n,n), $ we deduce, thanks to \eqref{R1} that  $$\{u_{n}\}_{n\geq2}\to u\quad \text{in}\quad L^{2}_{loc}(\mathbb{R}^{2})$$ with 
\begin{equation*}
\supp \ u \subset [t_{1} - \epsilon, t_{2} + \epsilon]\times{\mathbb{R}}
\end{equation*}
due to the fact \eqref{4.26}. Additionally,  $Pu= f$ in $\mathbb{R}^{2}$ by the third equation of \eqref{R}. Thus, the proof is finished.
\end{proof}

\section{Approximation Theorem applied in control problem}\label{Sec4} 
In this section, we present a direct application of the approximation Theorem \ref{main2}, which ensures the proof of Theorem \ref{main3}.
\subsection{Proof of Theorem \ref{main3}}
As is well know,  see \cite{MC}, that there exist $u_{1}$ and $u_{2}$ in a class $ C(0,T; H^{s}(0,+\infty)$, for $s \in\left(-\frac{7}{4}, \frac{5}{2}\right) \backslash\left\{\frac{1}{2}, \frac{3}{2}\right\}$, solutions of (without specification of the boundary conditions)
\begin{equation*}
\begin{cases}
u_{1t} +  u_{1x} +  u_{1xxx} - u_{1xxxxx}= 0 & \mbox{in} \ (0,T)\times(0,+\infty), \\ 
u_{1}(0,x)= u_{0} & \mbox{in} \ (0,+\infty)
\end{cases}
\end{equation*} 
and
\begin{equation*}
\begin{cases}
u_{2t} +  u_{2x} +  u_{2xxx} - u_{2xxxxx}= 0 & \mbox{in} \ (0,T)\times(0,+\infty), \\ 
u_{2}(0,x)= u_{T} & \mbox{in} \ (0,+\infty),
\end{cases}
\end{equation*} 
respectively,  for $s \in\left(-\frac{7}{4}, \frac{5}{2}\right)$.  Now, consider $\tilde{u}_{2}(t,x)= u_{2}(t-T,x)$. We have that $P\tilde{u}_{2}=0$ in $[0,T]\times(0,+\infty)$. Now, pick any $\epsilon' \in (\epsilon, \frac{T}{2})$ and consider the function $\varphi \in C^{\infty}(0,T)$ defined by
\begin{equation*}
\varphi(t)=\begin{cases}
1 ,&
\text{if} \ t \in [0, \epsilon']\\
0,& \text{if} \ t \in [T - \epsilon', T].
\end{cases}
\end{equation*}
Note that the change of variable
\begin{equation*}
u(t,x)=   \varphi(t)u_{1}(t,x) + (1- \varphi(t))\tilde{u}_{2}(t,x) + \omega(t,x),
\end{equation*}
transforms \eqref{PP} in
\begin{equation*}
\begin{cases}
\omega_{t} +  \omega_{x} +  \omega_{xxx} - \omega_{xxxxx}= \frac{d}{dt}\varphi(\tilde{u}_{2} - u_{1}) & \mbox{in} \ \mathcal{D}'((0,T)\times(0,+\infty)), \\ 
\omega(0,x)= \omega(T,x)=0 & \mbox{in} \ (0,+\infty).
\end{cases}
\end{equation*} 
The proof is finished taking into account the Corollary \ref{Prop princ} with $f= \frac{d\varphi}{dt}(\tilde{u}_{2} - u_{1})$. \qed

\subsection*{Acknowledgments:} The authors are grateful to the anonymous reviewers for the constructive comments that improved this work.
This work was done while the first author was visiting Virginia Tech. The author thanks the host institution for their warm hospitality.

\end{document}